\documentclass[10pt]{amsart}
\usepackage{amsmath,amssymb,xypic,array}
\usepackage[T1]{fontenc}
\usepackage{amsfonts,amsthm,epsfig}
\usepackage{setspace}
\usepackage{graphicx}
\usepackage{tikz}
\usepackage{booktabs}
\usepackage{longtable}
\usepackage{geometry}
\usetikzlibrary{trees}
\usepackage{tikz-cd}
\usepackage[english]{babel}
\usepackage[latin1]{inputenc}
\usepackage{color}
\usepackage{mathrsfs}
\usepackage{latexsym}
\usepackage{hyperref}
\usepackage{textcomp}

\newtheorem{Theorem}{Theorem}
\newtheorem{Lemma}[Theorem]{Lemma}
\newtheorem{Prop}[Theorem]{Proposition}
\newtheorem{Def}[Theorem]{Definition}

\newtheorem{Ex}[Theorem]{Example}

\newtheorem*{thm-intro}{Main Theorem}

\newcommand{\cali}{{\mathcal I}}
\newcommand{\calr}{{\mathcal R}}
\newcommand{\call}{{\mathcal L}}
\newcommand{\calm}{{\mathcal M}}

\def\C{{\mathbb{C}}}

\def\a{{\alpha}}
\def\b{{\beta}}

\newcommand{\op}[1]{{\mathcal O}_{\mathbb{P}^{#1}}}
\newcommand{\p}[1]{{\mathbb{P}^{#1}}}

\DeclareMathOperator{\supp}{Supp}
\DeclareMathOperator{\coker}{{coker}}
\DeclareMathOperator{\im}{{im}}

\DeclareMathOperator{\codim}{{codim}}

\title[Instanton sheaves and representations of quivers]
{Instanton sheaves and representations of quivers}

\author[  M. Jardim]{M. Jardim}
\thanks{ }
\dedicatory{}
\address{IMECC - UNICAMP \\ Departamento de Matem\'atica \\
Rua S\'ergio  Buarque de Holanda, 651\\ 13083-970 Campinas-SP, Brazil}
\email{jardim@ime.unicamp.br}

\author[ D. D. Silva  ]{ D. D. Silva }
\thanks{ }
\dedicatory{}
\address{DMA - UFS \\ Avenida Marechal Rondon S/N \\ S\~ao Cristov\~ao-SE, Brazil}
\email{ddsilva@ufs.br}

\date{}

\begin{document}

\begin{abstract}
We study the moduli space of rank 2 instanton sheaves on $\p3$ in terms of representations of a quiver consisting of 3 vertices and 4 arrows between two pairs of vertices. Aiming at an alternative compactification for the moduli space of instanton sheaves, we show that for each rank 2 instanton sheaf, there is a stability parameter $\theta$ for which the corresponding quiver representation is $\theta$-stable (in the sense of King), and that the space of stability parameters has a non trivial wall-and-chamber decomposition. Looking more closely at instantons of low charge, we prove that there are stability parameters with respect to which every representation corresponding to a rank 2 instanton sheaf of charge 2 is stable, and provide a complete description of the wall-and-chamber decomposition for representation corresponding to a rank 2 instanton sheaf of charge 1.
\end{abstract}

\maketitle

\section{Introduction}

Mathematical instanton bundles have been intensely studied by several authors since its introduction in the late 1970s by Atiyah, Drinfeld, Hitchin and Manin \cite{ADHM}. They arose as holomorphic counterparts, via twistor theory, to anti-self-dual connections with finite energy (instantons) on the four-dimensional round sphere, and can be defined as $\mu$-stable vector bundles $E$ on $\p3$ satisfying cohomological vanishing condition $h^1(E(-1))=0$ plus a \emph{reality} condition. A generalization to odd dimensional projective spaces was introduced by Okonek and Spindler in \cite{OS}, while a further generalization to non locally free sheaves on arbitrary projective spaces was considered in \cite{MJ}.   

In this paper, we will focus on rank 2 instantons sheaves on the 3-dimensional projective space. These can be defined as rank 2 torsion free sheaves $E$ on $\p3$ with trivial determinant and satisfying the vanishing conditions
$$h^0(E(-1))=h^1(E(-2))=h^2(E(-2))=h^3(E(-3))=0.$$
Let $n:=c_2(E)$, which is called the \emph{charge} of the instanton sheaf $E$; note that the vanishing conditions imply that $c_3(E)=0$. The moduli space $\cali(n)$ of rank 2 locally free instanton sheaves of charge $c$ is an affine \cite{CO}, irreducible \cite{T1,T2}, nonsingular \cite{JV} quasi-projective variety of the expected dimension $8n-3$. On the other hand, the moduli space $\call(n)$ of all rank 2 instanton sheaves has several irreducible components \cite{JMaiT,JMT2}, possibly of larger than expected dimension.

One can show that every rank 2 instanton sheaf is stable \cite[Theorem 4]{JMT2}, so the moduli spaces $\cali(n)$ and $\call(n)$ can be regarded as open subsets of the Gieseker--Maruyama moduli space $\calm(n)$ of rank 2 semistable sheaves with Chern classes $(c_1,c_2,c_3)=(0,n,0)$. An interesting problem, addressed in \cite{JMT1,MaTr2,NT,Per3} is to understand the closures $\overline{\cali(n)}$ and $\overline{\call(n)}$ of $\cali(n)$ and $\call(n)$ within the projective variety $\calm(n)$, and one remarkable fact is that both do contain locally free and non locally free sheaves which are not instanton when $n\ge2$. 

The key point of this paper is to present an alternative compactification of $\cali(n)$ and $\call(n)$ in terms of representations of quivers. Indeed, every instanton sheaf can be regarded as a representation of the following quiver
\begin{equation} \label{Q}
\mathbf{Q} =: \Biggl\{ \xymatrix{
\underset{-1}{\bullet} \ar@< 3pt> [r]^{\eta_0} \ar@<-9pt> [r]^\vdots_{\eta_{3}}	 &
\underset{0}{\bullet}  \ar@< 3pt> [r]^{\phi_0} \ar  @<-9pt> [r]^\vdots_{\phi_{3}}	 &
\underset{1}{\bullet} }\Biggr\}
\end{equation}
satisfying the relations $\phi_j\eta_i+\phi_i\eta_j=0$, with $0\le i,j \le 3$, plus additional open conditions, see details in Section \ref{prelim} below. One can then consider the projective moduli space of $\theta$-semistable representations of $\mathbf{Q}$ as constructed by King \cite{K}.

In this context, King's $\theta$-stability for representations of the quiver \eqref{Q} depends on two real parameters, and we obtain a wall-and-chamber decomposition of the real plane of stability parameters. One can then study where the representations corresponding to instanton sheaves are $\theta$-stable with respect to different stability parameters, and consider the compactification of $\cali(n)$ and $\call(n)$ within the projective moduli space of $\theta$-semistable representations of $\mathbf{Q}$.

The goal of this paper is to give the first steps in this program, providing a full picture in the simplest case, of charge 1 instanton sheaves.

More precisely, we prove that the moduli space of $\theta$-stable representations of $\mathbf{Q}$ with dimension vector $(n,2+2n,n)$ and $\theta=(\alpha,-n(\alpha+\gamma)/(2n+2),\gamma)$, henceforth denoted by $\calr_{\theta}(n)$, is always empty away from the fourth quadrant in the $\alpha\gamma$-plane. Next, we show that for each instanton sheaf $E$ there are stability parameters $\alpha$ and $\gamma$ for which the representation of $\mathbf{Q}$ corresponding to $E$ is $\theta$-stable. In addition, the line $\alpha+\gamma=0$ is a wall that destabilizes every instanton representation corresponding to a non locally free instanton sheaf.

Furthermore, when $n=1,2$, we show that there are stability parameters $\alpha$ and $\gamma$ for which {\it every} instanton representation of $\mathbf{Q}$ is $\theta$-stable. Finally, we establish the following result, providing a full picture for the case $n=1$.

\begin{thm-intro}
Let $\calr_{\theta}(1)$ be the moduli space of semistable representations of vector dimension $(1,4,1)$ with $\theta=(\alpha,-(\alpha+\gamma)/4,\gamma)$. If $(\alpha,\gamma)$ is a value outside the fourth quadrant of the $\alpha\gamma$-plane then $\calr_{\theta}(1)$ is empty. Otherwise, the moduli space $\calr_{\theta}(1)$ is isomorphic to $\mathbb{P}^5$,  containing $\cali(1)$ as the complement of an irreducible quadric. The points of this quadric are the representations corresponding to non locally free instanton sheaves when $\gamma < -\alpha$, and to the perverse instanton sheaves dual to the non locally free instanton sheaves when $\gamma > -\alpha$.
\end{thm-intro}

\begin{figure}[h!]
\includegraphics[scale=0.7]{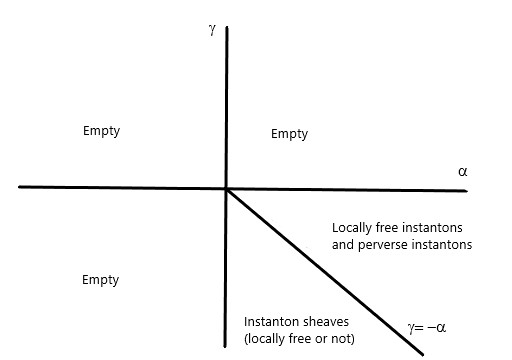}
\caption{The graph illustrates the Main Theorem, describing the points in the quiver moduli space $\calr_{\theta}(1)$ in each of the five regions of the $\alpha\gamma$-plane.}
\label{fig1}
\end{figure}

It is worth observing that an analogous picture was found in \cite{JMac} when considering instanton sheaves as objects in the derived category $D^b(\p3)$ of coherent sheaves on $\p3$. Indeed, one can find a point in the space of Brigdeland stability conditions on $\p3$ and an open neighbourhood of this point which is divided into 3 stability chambers whose moduli space of stable objects with Chern character $(2,0,-1,0)$ are exactly as described above: one chamber in which the moduli space is empty, one chamber in which the moduli space coincides with the moduli space of instanton sheaves of charge 1, and one chamber in which the moduli space consists of the locally free sheaves of charge 1 plus the dual of non locally free instanton sheaves of charge 1. Furthermore, it was also shown in \cite{JMac} that there is a Bridgeland stability condition in $D^b(\p3)$ with respect to which every instanton sheaf if stable.

This paper is organized as follows. We start by setting up notations and revising some key fact about instanton sheaves and representations of quivers in Section \ref{prelim}. We then prove the results for instanton representations of arbitrary charge mentioned above in Section \ref{instantons}. Finally Section \ref{ss:Sing} is dedicated to describing $\theta$-stable representations with the dimension vector of a representation corresponding to an instanton sheaf of charge 1 (see Theorem \ref{teorema}), later showing that there exist only one wall for this dimension vector in Section \ref{sec:modspc}, thus completing the proof of the Main Theorem. 

\subsection*{Acknowledgments}
The first named author is supported by the CNPQ grant number 302889/2018-3 and the FAPESP Thematic Project 2018/21391-1.
The second named author would like to thank IMECC-UNICAMP, the host institution of his postdoctoral position, and his own institution DMA-UFS for providing the necessary means for the research presented in this paper to be done. This work was also partially funded by CAPES - Finance Code 001.


\section{Preliminaries}\label{prelim}

We begin by setting up the notation and nomenclature to be used in the rest of the paper.

\subsection{Instanton sheaves}

\begin{Def}
An instanton sheaf on $\mathbb{P}^3$ is a torsion free sheaf $E$ on $\mathbb{P}^3$ with $c_1(E)=c_3(E)=0$ and satisfying 
$$h^0(E(-1))=h^1(E(-2))=h^2(E(-2))=h^3(E(-3))=0.$$
The charge of $E$ is given by its second Chern class $c_2(E)$.
\end{Def}

The definition above was originally proposed in \cite{MJ} in a broader context. In the present paper we only consider rank 2 instanton sheaves on $\mathbb{P}^3$.

Instanton sheaves are closed related to the concept of a linear monad by the use of the Beilinson spectral sequence. Recall that a {\it linear monad} on $\mathbb{P}^{3}$ is a complex of locally free sheaves of the form
\begin{equation} \label{linear monad}
\mathcal{O}_{\mathbb{P}^3}(-1)^{\oplus a} \stackrel{\alpha}{\rightarrow} \mathcal{O}_{\mathbb{P}^3}^{\oplus b} \stackrel{\beta}{\rightarrow} \mathcal{O}_{\mathbb{P}^3}(1)^{\oplus c}
\end{equation}
such that $\alpha$ is injective and $\beta$ is surjective. The sheaf $E:=\ker\beta/\im\alpha$ is called the cohomology of the monad. Consider the variety $\Sigma:=\supp(\coker\alpha^{*})$, which is called the {\it degeneration locus} of the monad. One can show that, see \cite[Proposition 4]{MJ}:
\begin{itemize}
\item [(i)] $E$ is torsion free if and only if $\codim\Sigma\ge2$;
\item [(ii)] $E$ is reflexive if and only if $\codim\Sigma=3$;
\item [(iii)] $E$ is locally free if and only if $\Sigma=\emptyset$.
\end{itemize}
Note that ${\rm rank}(E)=b-a-c$, $c_1(E)=c-a$, and $c_2(E)=(c+a+(c-a)^2)/2$.

A torsion free sheaf $E$ on $\mathbb{P}^3$ is said to be a {\it linear sheaf} if it can be represented as the cohomology of a linear monad. It can be proved that instanton sheaves on $\mathbb{P}^3$ are exactly the linear sheaves for which $c_1(E)=0$, that is, an instanton sheaf can be uniquely represented as the cohomology of a linear monad as in display \eqref{linear monad} for which $a=c$, see \cite[Proposition 2 and Thoerem 3]{MJ}. Therefore, rank 2 instanton sheaves of charge $n$ are in 1-1 correspondence with linear monads of the form
\begin{equation} \label{inst monad}
\mathcal{O}_{\mathbb{P}^3}(-1)^{\oplus n} \stackrel{\alpha}{\rightarrow} \mathcal{O}_{\mathbb{P}^3}^{\oplus 2n+2} \stackrel{\beta}{\rightarrow} \mathcal{O}_{\mathbb{P}^3}(1)^{\oplus n}
\end{equation}
whose degeneration locus has codimension at least 2.

It will also be important for us to consider the following more general objects, which were first introduced in \cite[Section 3.2]{HL}; see \cite[Definition 5.6]{HJV} for an alternative definition. Below, $\mathcal{H}^p$ denotes the ${\rm p}^{\rm th}$-cohomology sheaf of an object in $D^{b}(\p3)$, while $\mathbb{H}^p$ denote its ${\rm p}^{\rm th}$-hypercohomology group.

\begin{Def}
A {\it perverse instanton sheaf} on $\mathbb{P}^3$ is an object $C_{\bullet}$ in $D^{b}(\p3)$ with $c_1(C_{\bullet})=0$ satisfying the following conditions:
\begin{enumerate}
\item $\mathcal{H}^p(C_{\bullet})=0$ for $p\ne0,1$;
\item $\mathbb{H}^p(C_{\bullet}\otimes\op3(q))=0$ if $p+q<0$ when $p=0,1$ and $p+q\ge0$ when $p=2,3$;
\item the left derived functor $Lj^*C^{\bullet}$ is a sheaf object where $j:l \hookrightarrow \p3$ is the inclusion of a line $l$ in $\mathbb{P}^3$.
\end{enumerate}
\end{Def}

Note that every instanton sheaf is a perverse instanton as a sheaf object in $D^{\rm b}(\p3)$. In addition, it follows from the considerations in \cite[Section 2]{GJ} that the derived dual of a rank 2 instanton sheaf is also a perverse instanton sheaf. However, there are rank 2 perverse instanton sheaves which are not dual to a sheaf. 

One can show that $\mathcal{H}^0(C_{\bullet})$ is a torsion free sheaf and $\dim\mathcal{H}^1(C_{\bullet})=1$, see \cite[Corollary 3.16]{HL}. The rank of $C_{\bullet}$ is defined to be the rank of $\mathcal{H}^0(C_{\bullet})$; the charge of $C_{\bullet}$ is defined to be the second Chern class of $C_{\bullet}$, which coincides with $c_2(\mathcal{H}^0(C_{\bullet}))+{\rm mult}(\mathcal{H}^1(C_{\bullet}))$.

If $\mathcal{H}^0(C_{\bullet})=0$, then the sheaf $\mathcal{H}^1(C_{\bullet})$ is called a {\it rank 0 instanton sheaf}, see \cite{GJ,HL,JMaiT} for further details on such sheaves. 

Furthermore, observe that every complex of sheaves like the one in display \eqref{inst monad} is a rank 2 perverse instanton sheaf when regarded as an object in $D^{b}(\p3)$, provided $\codim\supp(\coker\alpha^{*})$ and $\codim\supp(\coker\beta)$ are both at least 2. Conversely, every rank 2 perverse instanton sheaf is canonically isomorphic (in $D^{b}(\p3)$) to a complex of sheaves as in display \eqref{inst monad} satisfying the latter property, see \cite[Lemma 3.15]{HL}.


\subsection{Representation of quivers}

Recall that a {\it quiver} Q is given by a finite set of vertices $Q_0$, a finite set of arrows $Q_1$ and two maps $h,t:Q_1 \rightarrow Q_0$ called head and tail, respectively. A {\it linear representation} of a quiver is given by $R=(\{V_i\}_{i \in Q_0}; \{f_{\alpha}\}_{\alpha \in Q_1})$ where $V_i$ is a $\mathbb{C}$-vector space and $f_{\alpha}:V_{t(\alpha)}\rightarrow V_{h(\alpha)}$ is linear. A morphism between two representations $R$ and $R'$ is given by $\phi=\{ \phi_{i} \}_{i \in Q_{0}}$ where $\phi_i: V_i \rightarrow V'_i$ is linear and for each arrow $\alpha$ we have $f'_{\alpha} \phi_{t(\alpha)}= \phi_{h(\alpha)}f_{\alpha}$. We denote ${\rm Rep}_{\mathbb{C}}Q$ the abelian category of the linear representations of the quiver $Q$.

The {\it algebra of the linear quiver} $Q$ is the associative $\mathbb{C}$-algebra $\mathbb{C}Q$ determined by generators $e_i$, where $i \in Q_0$, and $\alpha$, where $\alpha \in Q_1$ and the relations:

$e_ie_j=0$ if $i \neq j$, $e_i^2=e_i$, $e_{t(\alpha)}\alpha=\alpha e_{h(\alpha)}=\alpha$.

From the relations above, for any arrows $\alpha,\beta$ we get $\alpha \beta=0$ unless $h(\alpha)=t(\beta)$. Thus a product of arrows $\alpha_l \cdots \alpha_1$ is zero unless the sequence $\pi = (\alpha_1, \cdots, \alpha_l)$ is a {\it path}, i.e., $h(\alpha_j)=t(\alpha_{j+1})$ for $j=1, \cdots, l-1$. We then put $s(\pi)=s(\alpha_1)$, $t(\pi)=t(\alpha_l)$ and the ${\it length}$ of the path $\pi$, $l(\pi)=l$. For any vertex $i$ we also view $e_i$ as the {\it path of length} $0$ at the vertex $i$. 

Clearly the paths generate the vector space $\mathbb{C}Q$. They also are linearly independent: consider indeed the {\it path algebra} with basis the set of all paths and multiplication given by concatenation of paths. From the concept of a path algebra we get the following definition of quiver with relations generalizing the former definition of quiver:

\begin{Def}
A relation on a quiver $Q$ is a linear combination of paths in $\mathbb{C}Q$ having a common source and a common target and of length at least 2. A quiver with relations is a pair $(Q,I)$ where $Q$ is a quiver and $I$ is a two-sided ideal of $\mathbb{C}Q$ generated by relations. The quotient algebra $\frac{\mathbb{C}Q}{I}$ is the path algebra of $(Q,I)$.
\end{Def}

In this paper, we shall be interested in the quiver given in (1): 
\begin{center}
$ \mathbf{Q}:=$
\begin{tikzcd}
\stackrel{-1}{\circ} \arrow[r, "\eta_2" description , shift right=2]
  \arrow[r,"\eta_3" description, shift right=6]  
  \arrow[r, "\eta_1" description, shift left=2]
  \arrow[r, "\eta_0" description,  shift left=6]
& \stackrel{0}{\circ} \arrow[r, "\phi_2" description , shift right=2]
  \arrow[r,"\phi_3" description, shift right=6]  
  \arrow[r, "\phi_1" description, shift left=2]
  \arrow[r, "\phi_0" description,  shift left=6]
& \stackrel{1}{\circ} 
\end{tikzcd}
\end{center}
with relations $P_{ij}:=\phi_i\eta_j+\phi_j\eta_i=0$ for $0 \leq i \leq j \leq 3$.

A representation $R=(V_{-1},V_0,V_1;\{f_{\eta_i}\}, \{g_{\phi_i}\})$ of $\mathbf{Q}$ is said to satisfy the relations $P_{ij}$ when $g_{\phi_i}f_{\eta_j}+g_{\phi_j}f_{\eta_i}=0$.

\begin{Def}
Let $R=(V_{-1}, V_0, V_1, \{f_{\eta_i}\}, \{g_{\phi_j}\})$ be a representation of the quiver $\mathbf{Q}$ with relations $P_{ij}$. 
\begin{enumerate}
\item $R$ is {\it globally (locally) injective} if for every $(\lambda_0, \lambda_1, \lambda_2, \lambda_3) \in \mathbb{C}^4\setminus\{0\}$ (away from a subset of codimension at most 2), $\sum \lambda_i f_{\eta_i}$ is injective.
\item $R$ is {\it globally (locally) surjective} if for every $(\lambda_0, \lambda_1, \lambda_2, \lambda_3) \in \mathbb{C}^4\setminus\{0\}$ (away from a subset of codimension at most 2), $\sum \lambda_i g_{\eta_i}$ is surjective.
\item $R$ is an {\it instanton representation} if it is locally injective, globally surjective, and $\dim R=(n,2n+2,n)$ for some $n\ge0$, called the {\it charge} of $R$.
\item $R$ is a {\it perverse representation} if it is locally injective, locally surjective, and $\dim R=(n,2n+2,n)$ for some $n\ge0$, also called the {\it charge} of $R$.
\end{enumerate}
\end{Def}

We will make use of the following elementary facts:
\begin{enumerate}
\item If a representation $R$ with dimension vector $(a,b,c)$ is locally injective, then $b\ge a+1$;
\item If a representation $R$ with dimension vector $(a,b,c)$ is globally injective, then $b\ge c+3$;
\item every subrepresentation of a locally (globally) injective representation is also locally (globally) injective;
\item every quotient of a (locally) globally surjective representation is also (locally) globally surjective.
\end{enumerate}

\begin{Ex}
It is clear that a representation $R$ with $\dim R=(1,4,1)$ is globally injective if, and only if, $\{f_{\eta_0},f_{\eta_1},f_{\eta_2},f_{\eta_3}\}$ is a basis of ${\rm Hom}(\mathbb{C},\mathbb{C}^4)=\mathbb{C}^4$, while $R$ is globally surjective if, and only if, $\{g_{\phi_0},g_{\phi_1},g_{\phi_2},g_{\phi_3}\}$ is a basis of ${\rm Hom}(\mathbb{C}^4,\mathbb{C})=\mathbb{C}^4$. \qed
\end{Ex}


\subsection{Equivalence between categories of monads and representations}\label{subsec:eqv}

Let $\mathfrak{C}$ be the category of complexes of the form \eqref{linear monad}, regarded as a full subcategory of the category of complexes of sheaves on $\p3$. We shall also denote by $\mathfrak{Q}$ the abelian category of representations of $Q$ satisfying the relations $P_{ij}$.



\begin{Prop}\label{functor}
There is an equivalence of categories between between $\mathfrak{C}$ and $\mathfrak{Q}$. Moreover, under this equivalence:
\begin{enumerate}
\item instanton sheaves are in 1-1 correspondence with instanton representations of $\mathbf{Q}$;
\item perverse instanton sheaves which are dual to the instanton sheaves of the first item are in 1-1 correspondence with perverse representations of $\mathbf{Q}$;
\item locally free instanton sheaves are in 1-1 correspondence with instanton representations of $\mathbf{Q}$ that are globally injective.
\end{enumerate}
\end{Prop}

\begin{proof}
We construct an equivalence functor $\mathbf{F}$ between $\mathfrak{C}$ and $\mathfrak{Q}$ which restricts to the desired equivalences between their subcategories. Similar partial results in this direction were obtained in \cite{MJDP} and \cite{JV}.

First, fix homogeneous coordinates $[x_0:x_1:x_2:x_3]$ of $\p3$, and let $\{x_0,x_1,x_2,x_3\}$ be the corresponding basis of $H^0(\op3(1))$; one has a natural isomorphism
$$ {\rm Hom}(\op3(-1)^{\oplus a}, \op3^{\oplus b}) \simeq M_{b \times a}\otimes_{\mathbb{C}}H^0(\op3(1)) , $$
where $M_{b \times a}$ denotes the vector space of $b \times a$ matrices of complex numbers.

Consider the complex
$$ C_{\bullet}:\mathcal{O}_{\mathbb{P}^3}(-1)^{\oplus a} \stackrel{\alpha}{\rightarrow} \mathcal{O}_{\mathbb{P}^3}^{\oplus b} \stackrel{\beta}{\rightarrow} \mathcal{O}_{\mathbb{P}^3}(1)^{\oplus c} . $$
As $\alpha$ and $\beta$ can be seen as matrices whose entries are linear forms on $x_0,x_1,x_2,x_3$ we have
$$ \alpha= \alpha_0x_0 + \cdots + \alpha_3x_3 ~~,~~
\beta= \beta_0x_0 + \cdots + \beta_3x_3 ~~ , $$
where $\alpha_i \in M_{b \times a}$ and $\beta \in M_{c \times b}$. Hence we can set
$$F(C_{\bullet})=(\mathbb{C}^{a},\mathbb{C}^{b},\mathbb{C}^{c},\{\alpha_i\}, \{\beta_j\}).$$
Further, we have
$$\beta \circ \alpha=0 \Longleftrightarrow \sum_{i \leq j}(\beta_i\alpha_j+\beta_j\alpha_i)x_ix_j=0.$$
It follows that
$$\beta \circ \alpha=0 \Leftrightarrow \beta_i\alpha_j+\beta_j\alpha_i=0, 0 \leq i \leq j \leq 3.$$
Therefore, $\mathbf{F}(C_{\bullet})$ satisfies the relations of $\mathbf{Q}$. 

Given a morphism $\phi_{\bullet}:C_{\bullet} \rightarrow N_{\bullet}$ between complexes, by using the canonical isomorphism
${\rm Hom}(\op3(i)^{\oplus r},\op3(i)^{\oplus s}) \simeq M_{r \times s}$ where $i \in {\mathbb Z}$, we set $\mathbf{F}(\phi_{\bullet})$ to be the morphism of representations obtained from the above isomorphism.

Finally, the functor $\mathbf{F}$ is dense: given a representation in $\mathfrak{Q}$ and a choice of homogeneous coordinates for $\mathbb{P}^3$ one easily constructs a complex of the form \eqref{linear monad}. The functor is also faithfull and full since 
$${\rm Hom}_{\mathfrak{C}}(C_{\bullet}, D_{\bullet}) \stackrel{F}{\rightarrow}{\rm Hom}_{\mathfrak{Q}}(F(C_{\bullet}),F(D_{\bullet}))$$
is clearly an isomorphism.

For the second claim, just note that $\mathbf{F}(C_{\bullet})$ is locally injective if and only if the morphism $\alpha$ is injective, while $\mathbf{F}(C_{\bullet})$ is globally surjective if and only if the morphism $\beta$ is surjective. In addition, the degeneration locus of $C_{\bullet}$ is empty if and only if  $\mathbf{F}(C_{\bullet})$ is globally injective.
\end{proof}

To complete this section, recall that a representation $R$ of a quiver is said to be Schurian if every endomorphism is a multiple of the identity, that is ${\rm Hom}(R,R)\simeq\C$. Since every rank 2 instanton sheaf $E$ is simple (see \cite[Lemma 23]{MJ}), and the endomorphisms of $E$ bijective with the endomorphism of the corresponding monads \cite{OSS}, it follows from Proposition \ref{functor} that every instanton representation is Schurian. 



\subsection{Stability of representations}\label{subsec:stab}

Following King in \cite{K}, we consider the moduli space of representations of the quiver with relations $\mathbf{Q}$ of fixed dimension vector $(n,2n+2,n)$. Our notation and convention for the definition of semistability come from \cite{kirillov} though.  

Recall that for a quiver $\mathbf{Q}$ and a dimension vector $\bf{v} \in \mathbb{Z}^{I}_{+}$ where $I$ is the number of vertices of $\mathbf{Q}$ we define the representation space $R(\bf{v})=\bigoplus{\rm Hom}(\mathbb{C}^{\bf{v}_i}, \mathbb{C}^{\bf{v}_j})$ and the group ${\rm GL}(\bf{v})=\prod_{i \in \mathbf{Q_0}} {\rm GL}(\bf{v}_i)$ acting on it by conjugation. Since the group of constants acts trivially, we have an action of the group ${\rm PGL}(\bf{v})$ on the representation space. 
For the moduli space of representations we shall consider the twisted GIT quotient. Let $\theta \in \mathbb{Z}^{I}$ be a stability parameter and consider the character
$$\chi_{\theta}: {\rm GL}(\bf{v}) \rightarrow \mathbb{C}^{\times}$$

which sends $g$ to $\prod {\rm det}(g_i)^{-\theta_i}$. For the character to be well defined on $G={\rm PGL}(\bf{v})$, we must have

$$\theta \cdot {\bf v}= \sum_{i \in {\bf Q}_0} \theta_i \cdot {\bf v}_i=0.  $$

In this case, we define $\mathbb{C}[\bf{v}]^{G, \chi_{\theta}}=\{f \in \mathbb{C}[R(\bf{v})]:f(g \cdot m)= \chi_{\theta}(g)f(m)\}$, where $\mathbb{C}[R(\bf{v})]$ is the $\mathbb{C}$-algebra of regular functions on $R(\bf{v})$. Finally the GIT quotient associated to the stability parameter $\theta$ and to the dimension vector $\bf{v}$ is the variety:

$$\calr_{\theta}(\bf{v})=R(\bf{v})//_{\chi_{\theta}}G={\rm Proj}\left(\oplus_{n \geq 0} \mathbb{C}[\bf{v}]^{G,\chi_{\theta}^n} \right)$$

Now let $\theta \in {\mathbb R}^I$. A representation $V$ of $\mathbf{Q}$ is called $\theta$-semistable (respectively, $\theta$-stable) if $\theta \cdot {\bf dim}V=0$ and for any subrepresentation $V' \subset V$ we have $\theta \cdot {\bf dim}V' \leq 0$ (respectively, for every nonzero proper subrepresentation $V'$ we have $\theta \cdot {\bf dim}V' < 0$).

It was proved in \cite{K} that the GIT  $\chi_{\theta}$-semistable (respectively, $\chi_{\theta}$-stable) representations correspond to the $\theta$-semistable (respectively, $\theta$-stable) representations so we get the usual description of the moduli space $\calr_{\theta}(\bf{v})$ by means of $\theta$-semistable representations. 

In this paper we are interested in case ${\bf v}=(n,2n+2,n)$. We will set 
$$ \theta=\left( \alpha,-(\alpha+\gamma)\dfrac{n}{2n+2},\gamma \right), $$ 
so that $\theta\cdot(n,2n+2,n)=0$. From now on, we will denote by $\calr_{\theta}(n)$ the moduli space of $\theta$-semistable representations of $\mathbf{Q}$ with dimension vector $(n,2n+2,n)$ for $\theta$ as above.

A {\it stability chamber} is a subset $\Gamma$ of the $\alpha\gamma$-plane such that $\calr_{\theta_1}(n)=\calr_{\theta_2}(n)$ (as sets) for every $\theta_1,\theta_2\in\Gamma$. Each irreducible component of the complement of the union of all stability chambers is called a {\it wall}. Since $\theta$-stability is invariant under multiplication by a scalar (that is $\calr_{\theta}(n)=\calr_{\lambda\cdot\theta}(n)$ for every $\theta$ and every $\lambda\in\C^*$), it is easy to see that walls are lines passing through the origin of the $\alpha\gamma$-plane, while chambers are the unbounded regions limited by two such lines.


\section{Stability of instantons representations} \label{instantons}

Every representation $R$ of the quiver $\mathbf{Q}$ with $\dim R=(a,b,c)$ can be expressed as an extension of two other representations as follows
\begin{equation} \label{kernel}
0 \to K \to R \to A^{\oplus a} \to 0 ,
\end{equation}
where $\dim K=(0,b,c)$, and $A$ is the simple representation associated with the first vertex. With this in mind, $K$ is called the \emph{kernel subrepresentation} of $R$. Similarly, one also has a short exact sequence of the form
\begin{equation} \label{cokernel}
0 \to C^{\oplus c} \to R \to Q \to 0 ,
\end{equation}
where $\dim K=(a,b,0)$, and $C$ is the simple representation associated with the third vertex; $Q$ is called the \emph{cokernel quotient} of $R$. These previous 2 sequences correspond, under the functor $\mathbf{F}$ described in the proof of Proposition \ref{functor}, to the following short exact sequences of complexes:
$$ \xymatrix{
0\ar[d]  & 0\ar[d]  & 0\ar[d]  & ~~ & 
0\ar[d]  & 0\ar[d]  & 0\ar[d] \\
0\ar[r]\ar[d]  & \op3^{\oplus b}\ar[r]^\beta\ar[d]  & \op3(1)^{\oplus c}\ar[d]  & ~~ & 
0\ar[r]\ar[d]  & 0\ar[r]\ar[d]  & \op3(1)^{\oplus c}\ar[d] \\
\op3(-1)^{\oplus a} \ar[r]^\alpha\ar[d]  & \op3^{\oplus b}\ar[r]^\beta\ar[d]  & \op3(1)^{\oplus c}\ar[d]  & ~~ & 
\op3(-1)^{\oplus a}\ar[r]^\alpha\ar[d]  & \op3^{\oplus b}\ar[r]^\beta\ar[d]  & \op3(1)^{\oplus c}\ar[d] \\
\op3(-1)^{\oplus a}\ar[r]\ar[d]  & 0\ar[d]\ar[r]  & 0\ar[d]  & ~~ & 
\op3(-1)^{\oplus a}\ar[r]^\alpha\ar[d]  & \op3^{\oplus b}\ar[r]\ar[d]  & 0\ar[d] \\
0  & 0  & 0  & ~~ & 
0  & 0  & 0
}
$$

\begin{Lemma}\label{empty}
The moduli space $\calr_{\theta}(n)$ is empty whenever $(\alpha,\gamma)$ lies outside the fourth quadrant of the $\alpha\gamma$-plane.
\end{Lemma}
\begin{proof}
If $\alpha<0$, then $\theta\cdot\dim A^{\oplus n}=n\alpha<0$, so \eqref{kernel} is a destabilizing sequence for $R$. Similarly, if $\gamma>0$, then $\theta\cdot C^{\oplus n}=n\gamma>0$,  so \eqref{cokernel} is a destabilizing sequence for $R$.
\end{proof}

Next, we argue that there is a stability parameter $\theta$ for which the moduli space $\calr_{\theta}(n)$ is non-empty and contains (at least some) instanton sheaves, that is, $\call(n) \cap \calr_{\theta}(n) \neq \emptyset$.

\begin{Prop}
Let $R$ be an instanton representation. Then there exists a stability parameter $\theta$ for which $R$ is $\theta$-stable.
\end{Prop}

\begin{proof}
We already observed in the end of Section \ref{subsec:eqv} that every instanton representation $R$ is Schurian. In this case, the stabilizer group of $R$ is trivial and hence we get an open set in which the generic point has trivial stabilizer.
By a result of Van den Bergh \cite[Proposition 6]{BLB}, if the stabilizer group of $R$ is zero-dimensional then there is an invariant affine open set in which the generic orbit is closed. This open set in the GIT construction is given by the non-vanishing of a relative invariant function of some weight $\chi_{\theta}$. Hence the generic point will be $\chi_{\theta}$-stable and therefore $\theta$-stable by Theorem 4.1 in \cite{K}. Finally, as the conditions of locally injective and globally surjective are open we get the result. 
\end{proof}

Having proved that the moduli spaces $\calr_{\theta}(n)$ are not always trivial, we now show that there always are at least two different stability chambers within the fourth quadrant.

\begin{Lemma}
There is a wall that destabilizes all instanton representations corresponding to non locally free instanton sheaves, in any charge.
\end{Lemma}
\begin{proof}
Let $E$ be a non locally free rank 2 instanton sheaf of charge $n$, and let $R$ be the corresponding instanton representation. 

By the Main Theorem in \cite{GJ}, the double dual sheaf $E^{\vee\vee}$ is a locally free instanton sheaf, and $Q_E:=E^{\vee\vee}/E$ is a rank 0 instanton sheaf. Letting $Q_R$ and $S_R$ be the representations of $\mathbf{Q}$ corresponding to the sheaves $E^{\vee\vee}$ and $Q_E$, respectively, the short exact sequence of sheaves $0\to E\to E^{\vee\vee} \to Q_E\to 0$ gives rise to the short exact sequence $0\to S_R \to R \to Q_R \to 0$ in $\mathfrak{Q}$. Since $\dim S_R=(d,2d,d)$ for some $d\ge1$, we have that
$$\theta\cdot\dim S_R = \dfrac{d}{n+1} (\alpha+\gamma), $$
So $R$ is not $\theta$-semistable when $\alpha+\gamma>0$.

According to the previous proposition, there is a stability parameter $\theta$ for which $R$ is $\theta$-stable. Since $R$ cannot be $\theta$-semistable above the line $\alpha=-\gamma$, we obtain the desired statement.
\end{proof}

Of course, our goal is to know whether there exists a stability parameter $\theta$ for which {\it every} instanton representation is $\theta$-stable. In order to do that, one must find suitable restrictions on the possible dimension vectors of subrepresentations of instanton representations.

\begin{Lemma}\label{sub}
If $S$ is a nontrivial subrepresentation of an instanton representation of charge $n$ with $\dim S=(s_{-1},s_0,s_1)$, then the following inequalities hold:
\begin{enumerate}
\item $s_{-1}+1\le s_0$;
\item $s_0-s_1 \le n-1$ when $s_1<n$;
\item $s_0-4s_1 \le 0$;
\item $s_1\ge1$.
\end{enumerate}
\end{Lemma}
\begin{proof}
The first inequality simply reflects the fact that every subrepresentation $S$ of an instanton representation $R$  must be locally injective.

Similarly, the quotient representation $R/S$ must be globally surjective. Since 
$$ \dim R/S=(n-s_{-1},2n+2-s_0,n-s_1), $$ 
one must have, when $s_1<n$,
$$ (2n+2-s_0) - (n-s_1) \ge 3, $$
 which is equivalent to the inequality in item (2).

Next, consider the composed morphism $\phi:S\hookrightarrow R \twoheadrightarrow A^{\oplus n}$. It follows from the exact sequence in display \eqref{kernel} that $\ker\phi$ is a subrepresentation of the kernel subrepresentation of $R$, so in particular $\dim\ker\phi=(0,s_0,s_1)$. Thus $\ker\phi$ is associated, via the functor $\mathbf{F}$ of Proposition \ref{functor}, to a morphism of sheaves $\beta':\op3^{\oplus s_0} \to \op3(1)^{\oplus s_1}$. Note that $\ker\beta'$ is a subsheaf of $\ker\beta$, which has no global sections since $H^0(\ker\beta)=H^0(E)=0$. Therefore,  $H^0(\ker\beta')=0$ as well, which means that the induced map in cohomology
$$ H^0(\op3^{\oplus s_0}) \stackrel{H^0(\beta')}{\longrightarrow} H^0(\op3(1)^{\oplus s_1}) $$
must be injective, thus $s_0\le4s_1$, as desired.

Finally, if $s_1=0$, then the inequality in item (3) implies that $s_0=0$, while the first inequality implies that $s_{-1}=0$ as well.
\end{proof}

The inequalities in the previous lemma are all we need to answer our main question when $n\le2$. In fact, the case $n=1$ was already considered in \cite[Section 6]{MMS}, where it shown, in a broader context, there is $\theta$ for which that every representation corresponding to a locally free instanton of charge 1 is $\theta$-stable; we will say more about this case in Section \ref{ss:Sing} below. We close this section by considering the case $n=2$.

\begin{Prop}
There exists a stability parameter $\theta$ for which every instanton representation of charge 2 is $\theta$-stable.
\end{Prop}
\begin{proof}
We show that there exists $0 < \epsilon \ll 1$ for which every every instanton representation of charge 2 is $\theta_{\epsilon}$-stable, where $\theta_{\epsilon}=(\epsilon, (1-\epsilon)/3,-1)$. 
We have 
$$ \theta_\epsilon\cdot(s_{-1},s_0,s_1) = \left( s_{-1}-\dfrac{s_0}{3}\right)\epsilon + \dfrac{s_0}{3} - s_1. $$

By the fourth item in Lemma \ref{sub}, it is enough to consider  the cases $s_1=1,2$.
\begin{itemize}
\item Case $s_1=2$. If $s_0<6$, then $\dfrac{s_0}{3}-s_1<0$, hence, since the quantity inside the first parenthesis can only have finitely many values, one can find $0 < \epsilon \ll 1$ for which $\theta \cdot (s_{-1},s_0,2)<0$. If $s_0=6$, then $s_{-1}\le1$, so again $\theta \cdot (s_{-1},6,2)<0$.
\item Case $s_1=1$. by item (2) of Lemma \ref{sub} we have $s_0 \leq 2$ and hence $s_0/3-1<0$. Again one can find $\epsilon$ for which $\theta \cdot (s_{-1},s_0,2)<0$.
\end{itemize}
\end{proof}


\section{Description of representations in $\cali(1)$} \label{ss:Sing}

We consider again the quiver with relations $\mathbf{Q}$ and representations of this quiver with vector dimension $(1,4,1)$. If $\theta=(\alpha, \beta, \gamma)$ is a stability parameter then as $\theta \cdot (1,4,1)=0$ we get $\theta=(\alpha,-(\alpha+\gamma)/4,\gamma)$. Finally let $\calr_{\theta}(1)$ be the moduli space of semistable representations of the quiver $\mathbf{Q}$ of fixed dimension vector $(1,4,1)$. We want to establish conditions on $\alpha$ and $\gamma$ in order to get $\cali(1) \subset \calr_{\theta}(1)$ as we know that $\cali(1)$ may be seen in $\calr_{\theta}(1)$ as the set of orbits of representations which are globally surjective and globally injective.

From now on we shall use the notation $\mathbb{C}^{b}=0$ if $b=0$.

\begin{Prop}\label{existence0b1}
Every representation $R$ in $\calr_{\theta}(1)$ has subrepresentation $S$ of dimension vector $(0,b,1)$ for all $b \in \{0,1,2,3,4\}$.

\end{Prop}
\begin{proof}
Let $R \in \calr_{\theta}(1)$ be the representation given by
\begin{center}
$R:$
\begin{tikzcd}
\mathbb{C} \arrow[r, "u_2" description , shift right=2]
  \arrow[r,"u_3" description, shift right=6]  
  \arrow[r, "u_1" description, shift left=2]
  \arrow[r, "u_0" description,  shift left=6]
& \mathbb{C}^{4} \arrow[r, "v_2" description , shift right=2]
  \arrow[r,"v_3" description, shift right=6]  
  \arrow[r, "v_1" description, shift left=2]
  \arrow[r, "v_0" description,  shift left=6]
& \mathbb{C}
\end{tikzcd}
\end{center}
From the decomposition $\mathbb{C}^4=\mathbb{C}^b\oplus\mathbb{C}^{4-b}$ (being trivial in case $\mathbb{C}^{b}=0$ or $\mathbb{C}^{b}=\mathbb{C}^4$) we can define for all $i \in \{0,1,2,3\}$,  $v'_i=v_ij$ where $j: \mathbb{C}^{b} \hookrightarrow \mathbb{C}^{4}$ is the inclusion in the first summand of the decomposition. It is clear that the representation

\begin{center}
$S:$
\begin{tikzcd}
 0 \arrow[r, shift right=2]
  \arrow[r, shift right=6]  
  \arrow[r, shift left=2]
  \arrow[r, shift left=6]
& \mathbb{C}^{b} \arrow[r, "v'_2" description , shift right=2]
  \arrow[r,"v'_3" description, shift right=6]  
  \arrow[r, "v'_1" description, shift left=2]
  \arrow[r, "v'_0" description,  shift left=6]
& \mathbb{C}
\end{tikzcd}
\end{center}
satisfies the relations of $\mathbf{Q}$ and it is a subrepresentation of $R$:
\begin{center}
\begin{tikzcd}
   0 \arrow[r, shift right=2]
  \arrow[r, shift right=6]  
  \arrow[r, shift left=2]
  \arrow[r, shift left=6]
  \arrow[dd]
& \mathbb{C}^{b} \arrow[r, "v'_2" description , shift right=2]
  \arrow[r,"v'_3" description, shift right=6]  
  \arrow[r, "v'_1" description, shift left=2]
  \arrow[r, "v'_0" description,  shift left=6]
  \arrow[dd,"j"]
& \mathbb{C} \arrow[dd, "1_{\mathbb{C}}"]\\
&&\\
\mathbb{C} \arrow[r, "u_2" description , shift right=2]
  \arrow[r,"u_3" description, shift right=6]  
  \arrow[r, "u_1" description, shift left=2]
  \arrow[r, "u_0" description,  shift left=6]
& \mathbb{C}^{4} \arrow[r, "v_2" description , shift right=2]
  \arrow[r,"v_3" description, shift right=6]  
  \arrow[r, "v_1" description, shift left=2]
  \arrow[r, "v_0" description,  shift left=6]
& \mathbb{C}
\end{tikzcd}
\end{center}
being the second square commutative from the expression $v'_i=v_ij$ for all $i \in \{0,1,2,3\}$. 
\end{proof}

We are interested in knowing the possible dimension vectors of subrepresentations of a representation $R \in \cali(1)$. For this we have to study the globally surjective and the globally injective representations in more detail.

\begin{Def}
Let $R$ be the representation of the quiver with relations $\mathbf{Q}$
\begin{center}
{\rm (I)}   $R:$
\begin{tikzcd}
\mathbb{C} \arrow[r, "u_2" description , shift right=2]
  \arrow[r,"u_3" description, shift right=6]  
  \arrow[r, "u_1" description, shift left=2]
  \arrow[r, "u_0" description,  shift left=6]
& \mathbb{C}^{4} \arrow[r, "v_2" description , shift right=2]
  \arrow[r,"v_3" description, shift right=6]  
  \arrow[r, "v_1" description, shift left=2]
  \arrow[r, "v_0" description,  shift left=6]
& \mathbb{C}
\end{tikzcd}
\end{center}
We say $R$ is {\it globally surjective of rank $r$} if $R$ is globally surjective and the rank of the matrix $M=[u_0,u_1,u_2,u_3]$ where the $u_i$ are the column vectors of $M$ is equal to $r$. Similarly, we say $R$ is {\it globally injective of rank $r$} if $R$ is globally injective and the rank of the matrix $N=[v_0,v_1,v_2,v_3]^{T}$ where the $v_i$ are the row vectors of $N$ is equal to $r$.  
\end{Def}

{\it Remark:} It is clear that we could have changed the roles of row and column vectors or used just one of them in the above definition but the notation introduced here will be important to what follows.

Now we are going to explain a few facts that shall be used throughout the rest of the paper. 
Let again $R$ be the representation as in (I) such that the $4 \times 4$ matrix $N=[v_0, v_1, v_2, v_3]^{T}$, where the row vectors are the vectors $v_i \in {\rm Hom}(\mathbb{C}^4,\mathbb{C})=\mathbb{C}^4$, has rank $b \in \{0,1,2,3,4\}$.
Then we take $g=(1,A^{-1},1)$ where $A \in {\rm Gl}(\mathbb{C}^4)$ is the invertible matrix that we multiply $N$ on the right in order to get a matrix of the kind

$$ \tilde{N}=\begin{bmatrix}
      I_{b} & 0\\
      *  & 0
\end{bmatrix}$$   

where $I_{b}$ represents the identity matrix of order $b$ and $*$ represents a possible nontrivial submatrix of order $(4-b) \times b$. Acting $g$ on $R$ we get a representation 

\begin{center}
 $\tilde{R}:$
\begin{tikzcd}
\mathbb{C} \arrow[r, "\tilde{u}_2" description , shift right=2]
  \arrow[r,"\tilde{u}_3" description, shift right=6]  
  \arrow[r, "\tilde{u}_1" description, shift left=2]
  \arrow[r, "\tilde{u}_0" description,  shift left=6]
& \mathbb{C}^{4} \arrow[r, "\tilde{v}_2" description , shift right=2]
  \arrow[r,"\tilde{v}_3" description, shift right=6]  
  \arrow[r, "\tilde{v}_1" description, shift left=2]
  \arrow[r, "\tilde{v}_0" description,  shift left=6]
& \mathbb{C}
\end{tikzcd}
\end{center}
in the same orbit of $R$ such that $\tilde{N}=[\tilde{v}_0, \tilde{v}_1, \tilde{v}_2, \tilde{v}_3]^{T}$, i.e., $\tilde{v}_0, \tilde{v}_1, \tilde{v}_2, \tilde{v}_3$ are the rows of $\tilde{N}$.
It is clear that the sets of dimension vectors of subrepresentations of $R$ and ${\tilde R}$ are the same. 

Observe that in the special case $b=4$ we get in the same orbit of $R$ a representation with the canonical basis of $\mathbb{C}^4$ in the places of $\{v_0, v_1, v_2, v_3\}$. 

Analogously, given a representation $R$ such that the $4\times4$ matrix $M=[u_0,u_1,u_2,u_3]$, where the column vectors are the vectors $u_i \in {\rm Hom}(\mathbb{C},\mathbb{C}^4)=\mathbb{C}^4$, has rank $b \in \{0,1,2,3,4\}$ we can take $g=(1,A,1)$ where $A$ is the invertible matrix that we multiply $M$ on the left in order to find a matrix of the kind:

$$ \tilde{M}=\begin{bmatrix}
      I_{b} & *\\
      0  & 0
\end{bmatrix}$$   

where $I_{b}$ represents the identity matrix of order $b$ and $*$ represents a possible nontrivial submatrix of order $b \times (4-b)$. Acting $g$ on $R$ we get a representation 

\begin{center}
 $\tilde{R}:$
\begin{tikzcd}
\mathbb{C} \arrow[r, "\tilde{u}_2" description , shift right=2]
  \arrow[r,"\tilde{u}_3" description, shift right=6]  
  \arrow[r, "\tilde{u}_1" description, shift left=2]
  \arrow[r, "\tilde{u}_0" description,  shift left=6]
& \mathbb{C}^{4} \arrow[r, "\tilde{v}_2" description , shift right=2]
  \arrow[r,"\tilde{v}_3" description, shift right=6]  
  \arrow[r, "\tilde{v}_1" description, shift left=2]
  \arrow[r, "\tilde{v}_0" description,  shift left=6]
& \mathbb{C}
\end{tikzcd}
\end{center}
where $\tilde{u}_0, \tilde{u}_1, \tilde{u}_2, \tilde{u}_3$ are the column vectors of the matrix $\tilde{M}$. Again, in case $b=4$ we have that $\{\tilde{u}_0,\tilde{u}_1,\tilde{u}_2,\tilde{u}_3\}$ is the canonical basis of $\mathbb{C}^4$ and we also have that the sets of dimension vectors of subrepresentations of $R$ and $\tilde{R}$ are the same.

We are going to use the discussion above to get a characterization of both globally surjective representations and globally injective representations in terms of the dimension vectors of their subrepresentations.

\begin{Theorem}\label{teorema}
Let $R$ be a representation in $\calr_{\theta}(1)$. Then
\begin{enumerate}
\item $R$ is globally injective if, and only if, there does not exist subrepresentation $S$ of $R$ of dimension vector $(1,b,1)$ for $b \in \{0,1,2,3\}$. 
\item $R$ is globally surjective if, and only if, there does not exist subrepresentation $S$ of $R$ of dimension vector $(0,b,0)$ for $b \in \{1,2,3,4\}$. 
\end{enumerate}
\end{Theorem}
\begin{proof}
Let $R$ be as in (I). 

Item 1).

Suppose $R$ is globally injective. Then $\{u_0,u_1,u_2,u_3\}$ is a basis for ${\rm Hom}(\mathbb{C},\mathbb{C}^4)=\mathbb{C}^4$. Let us prove the implication by contradiction.

If $S$ is a subrepresentation of dimension vector $(1,b,1)$, for $b<4$, then we get the quotient $R/S$ as below
\begin{center}
\begin{tikzcd}
   \mathbb{C} \arrow[r, "u_2" description , shift right=2]
  \arrow[r, "u_3" description, shift right=6]  
  \arrow[r, "u_1" description, shift left=2]
  \arrow[r, "u_0" description, shift left=6]
  \arrow[dd]
& \mathbb{C}^{4} \arrow[r, "v_2" description , shift right=2]
  \arrow[r,"v_3" description, shift right=6]  
  \arrow[r, "v_1" description, shift left=2]
  \arrow[r, "v_0" description,  shift left=6]
  \arrow[dd,"p"]
& \mathbb{C} \arrow[dd]\\
&&\\
0 \arrow[r, shift right=2]
  \arrow[r, shift right=6]  
  \arrow[r, shift left=2]
  \arrow[r, shift left=6]
& \mathbb{C}^{4-b} \arrow[r, shift right=2]
  \arrow[r, shift right=6]  
  \arrow[r, shift left=2]
  \arrow[r, shift left=6]
& 0
\end{tikzcd}
\end{center}
The kernel of the map $p$ has dimension $b<4$ and from the diagram above we see that $\{u_0,u_1,u_2,u_3\}$ is contained in it. As $\{u_0,u_1,u_2,u_3\}$ are linearly independent we have a contradiction.

Now suppose $R$ is not globally injective and suppose the rank of the matrix $M=[u_0,u_1,u_2,u_3]$ is $b<4$. From the discussion above we can consider $R$ in such a way that the matrix $M$ is of the kind
 $$ M=\begin{bmatrix}
      I_{b} & *\\
      0  & 0
\end{bmatrix}$$
where $u_0, u_1, u_2, u_3$ are the column vectors of the matrix $M$.

We show that there exists subrepresentation $S$ of vector dimension $(1,b,1)$. Indeed, let $S$ be the representation denoted by
\begin{center}
$S:$
\begin{tikzcd}
\mathbb{C} \arrow[r, "u'_2" description , shift right=2]
  \arrow[r,"u'_3" description, shift right=6]  
  \arrow[r, "u'_1" description, shift left=2]
  \arrow[r, "u'_0" description,  shift left=6]
& \mathbb{C}^{b} \arrow[r, "v'_2" description , shift right=2]
  \arrow[r,"v'_3" description, shift right=6]  
  \arrow[r, "v'_1" description, shift left=2]
  \arrow[r, "v'_0" description,  shift left=6]
& \mathbb{C}
\end{tikzcd}
\end{center}
where $\{u'_0,u'_1,u'_2,u'_3\}$ are the column vectors of the submatrix $M'$ of $M$ given by the first $b$ rows of $M$ 
$$M'=
\begin{bmatrix}
I_{b} & *
\end{bmatrix}
$$

and $\{v'_0,v'_1,v'_2,v'_3\}$ are the row vectors of the submatrix $N'$ of $N=[v_0,v_1,v_2,v_3]^{T}$ (where each $v_i$ is a row vector) given by the first $b$ columns of $N$. Later we are going to show that $S$ also satisfies the relations of the quiver $\mathbf{Q}$.

Consider the map $\phi: \mathbb{C}^{b} \rightarrow \mathbb{C}^{4}$ given by

$$\phi=
\begin{bmatrix}
I_{b}\\
0
\end{bmatrix}
$$
where $I_{b}$ is the identity matrix of order $b$. We need to show that the diagram below commute: 
\begin{center}
\begin{tikzcd}
   \mathbb{C} \arrow[r, "u'_2" description , shift right=2]
  \arrow[r, "u'_3" description, shift right=6]  
  \arrow[r, "u'_1" description, shift left=2]
  \arrow[r, "u'_0" description, shift left=6]
  \arrow[dd, "1_{\mathbb{C}}"]
& \mathbb{C}^{b} \arrow[r, "v'_2" description , shift right=2]
  \arrow[r,"v'_3" description, shift right=6]  
  \arrow[r, "v'_1" description, shift left=2]
  \arrow[r, "v'_0" description,  shift left=6]
  \arrow[dd,"\phi"]
& \mathbb{C} \arrow[dd, "1_{\mathbb{C}}"]\\
&&\\
\mathbb{C} \arrow[r,"u_2" description, shift right=2]
  \arrow[r, "u_3" description, shift right=6]  
  \arrow[r, "u_1" description, shift left=2]
  \arrow[r, "u_0" description, shift left=6]
& \mathbb{C}^{4} \arrow[r, "v_2" description, shift right=2]
  \arrow[r, "v_3" description, shift right=6]  
  \arrow[r, "v_1" description, shift left=2]
  \arrow[r, "v_0" description, shift left=6]
& \mathbb{C}
\end{tikzcd}
\end{center}
We have $\phi u'_i=u_i$ for all $i \in \{1,2,3,4\}$ since 

$$\phi \cdot M'=
\begin{bmatrix}
I_{b}\\
0
\end{bmatrix}
\begin{bmatrix}
I_{b} & *
\end{bmatrix}
=
\begin{bmatrix}
I_{b} & *\\
0 & 0
\end{bmatrix}
=M
$$

and also we have $v_i\phi=v'_i$ for all $i \in \{0,1,2,3\}$ by the definition of the $v'_i$ themselves.

Observe that $S$ obeys the relations of the quiver $\mathbf{Q}$:

$$v'_iu'_j + v'_ju'_i=v_i \phi u'_j+ v_j \phi u'_i=v_iu_j+v_ju_i=0$$

for $0 \leq i \leq j \leq 3$.

Hence $S$ is a subrepresentation of $R$ of dimension vector $(1,b,1)$ with $b<4$.

Item 2).

Suppose $R$ is globally surjective. Then we know we may consider $\{v_0,v_1,v_2,v_3\}$ as the canonical basis of $\mathbb{C}^4={\rm Hom}(\mathbb{C}^4, \mathbb{C})$.

If there exists a subrepresentation $S$ of dimension vector $(0,b,0)$ where $b \in \{1,2,3,4\}$ then by the diagram 
\begin{center}
\begin{tikzcd}
   0 \arrow[r, shift right=2]
  \arrow[r, shift right=6]  
  \arrow[r, shift left=2]
  \arrow[r, shift left=6]
  \arrow[dd]
& \mathbb{C}^{b} \arrow[r, shift right=2]
  \arrow[r, shift right=6]  
  \arrow[r, shift left=2]
  \arrow[r, shift left=6]
  \arrow[dd,"j"]
& 0 \arrow[dd]\\
&&\\
\mathbb{C} \arrow[r,"u_2" description, shift right=2]
  \arrow[r, "u_3" description, shift right=6]  
  \arrow[r, "u_1" description, shift left=2]
  \arrow[r, "u_0" description, shift left=6]
& \mathbb{C}^{4} \arrow[r, "v_2" description, shift right=2]
  \arrow[r, "v_3" description, shift right=6]  
  \arrow[r, "v_1" description, shift left=2]
  \arrow[r, "v_0" description, shift left=6]
& \mathbb{C}
\end{tikzcd}
\end{center}
we get $v_i j=0$ for all $i \in \{0,1,2,3 \}$. But this implies $j=0$ and hence $\mathbb{C}^b=0$ which is a contradiction.

On the other hand suppose $R$ is not globally surjective and let the rank of the matrix $N=[v_0,v_1,v_2,v_3]^{T}$, where $v_0,v_1,v_2,v_3$ are the row vectors of $N$, be $b'<4$. Then we may consider $R$ such that the matrix $N$ is of the form

$$N=\begin{bmatrix} I_{b'} & 0\\ * & 0 \end{bmatrix}$$.

Set $b=4-b'$ with $b' \in \{0,1,2,3\}$. We shall prove that there exists subrepresentation $S$ of dimension vector $(0,b,0)$.

Consider the representation $S$
 \begin{center}
\begin{tikzcd}
0 \arrow[r, shift right=2]
  \arrow[r, shift right=6]  
  \arrow[r, shift left=2]
  \arrow[r, shift left=6]
& \mathbb{C}^{b} \arrow[r, shift right=2]
  \arrow[r, shift right=6]  
  \arrow[r, shift left=2]
  \arrow[r, shift left=6]
& 0
\end{tikzcd}
\end{center}
which trivially satisfies the relations of the quiver $\mathbf{Q}$. We take the injective map $\phi: \mathbb{C}^{b} \rightarrow \mathbb{C}^4$ given by matrix

$$\phi= \begin{bmatrix} 0\\ I_{b}  \end{bmatrix}$$

From equation 
$$
N \cdot \phi=
\begin{bmatrix}
I_{b'} & 0 \\ 
* & 0
\end{bmatrix}
\begin{bmatrix}
0\\
I_{b}
\end{bmatrix}
=
\begin{bmatrix}
0\\
0
\end{bmatrix}
$$

we get $v_i\phi=0$ for all $i \in \{0,1,2,3 \}$ which implies that $S$ is in fact a subrepresentation of $R$:

\begin{center}
\begin{tikzcd}
   0 \arrow[r, shift right=2]
  \arrow[r, shift right=6]  
  \arrow[r, shift left=2]
  \arrow[r, shift left=6]
  \arrow[dd]
& \mathbb{C}^{b} \arrow[r, shift right=2]
  \arrow[r, shift right=6]  
  \arrow[r, shift left=2]
  \arrow[r, shift left=6]
  \arrow[dd,"\phi"]
& 0 \arrow[dd]\\
&&\\
\mathbb{C} \arrow[r,"u_2" description, shift right=2]
  \arrow[r, "u_3" description, shift right=6]  
  \arrow[r, "u_1" description, shift left=2]
  \arrow[r, "u_0" description, shift left=6]
& \mathbb{C}^{4} \arrow[r, "v_2" description, shift right=2]
  \arrow[r, "v_3" description, shift right=6]  
  \arrow[r, "v_1" description, shift left=2]
  \arrow[r, "v_0" description, shift left=6]
& \mathbb{C}
\end{tikzcd}
\end{center}

\end{proof}

\begin{Prop}\label{posto0}
Let $R$ be a representation in $\calr_{\theta}(1)$ which is globally (surjective) injective. Then $R$ is not locally (injective) surjective if, and only if, $R$ has subrepresentation $S$ of dimension vector $(1,b,0)$ with $b \in \{0,1,2,3,4\}$.
\end{Prop}
\begin{proof}
Firstly, suppose $R$ globally injective. Let $S$ be a subrepresentation of $R$ of dimension vector $(1,b,0)$ with $b \in \{0,1,2,3,4\}$:
\begin{center}
(II)\begin{tikzcd}
   \mathbb{C} \arrow[r, "u'_2" description , shift right=2]
  \arrow[r, "u'_3" description, shift right=6]  
  \arrow[r, "u'_1" description, shift left=2]
  \arrow[r, "u'_0" description, shift left=6]
  \arrow[dd]
& \mathbb{C}^{b} \arrow[r, "v'_2" description , shift right=2]
  \arrow[r,"v'_3" description, shift right=6]  
  \arrow[r, "v'_1" description, shift left=2]
  \arrow[r, "v'_0" description,  shift left=6]
  \arrow[dd,"\phi"]
& 0 \arrow[dd]\\
&&\\
\mathbb{C} \arrow[r,"u_2" description, shift right=2]
  \arrow[r, "u_3" description, shift right=6]  
  \arrow[r, "u_1" description, shift left=2]
  \arrow[r, "u_0" description, shift left=6]
& \mathbb{C}^{4} \arrow[r, "v_2" description, shift right=2]
  \arrow[r, "v_3" description, shift right=6]  
  \arrow[r, "v_1" description, shift left=2]
  \arrow[r, "v_0" description, shift left=6]
& \mathbb{C}
\end{tikzcd}
\end{center}
From the diagram (II), as $\{u_0, u_1, u_2, u_3 \}$ is a basis of ${\rm Hom}(\mathbb{C},\mathbb{C}^4)=\mathbb{C}^4$ we get $b=4$ and hence $\phi$ is an isomorphism. Then from $v_i\phi=0$ we get $v_i=0$ for all $i \in \{0,1,2,3\}$, so the rank of $R$ is zero and $R$ is not locally surjective.
On the other hand, if $R$ is not locally surjective then $v_0=v_1=v_2=v_3=0$ and hence $R$ has subrepresentation $S$ of dimension vector $(1,4,0)$: using the notation of the diagram (II) it is enough to set $\phi=1_{\mathbb{C}^4}$ and $u'_i=v'_i$ for all $i$.

Now take $R$ globally surjective. Let $\{ v_0, v_1, v_2,v_3\}$ be the canonical basis of ${\rm Hom}(\mathbb{C}^4,\mathbb{C})=\mathbb{C}^4$. Let $S$ be a subrepresentation of $R$ of vector dimension $(1,b,0)$ where $b \in \{0,1,2,3,4\}$. By the diagram (II), we get $v_i\phi=0$ for all $i \in \{0,1,2,3\}$ and hence $\phi=0$. Thus from the same diagram we have $u_0=u_1=u_2=u_3=0$, i.e., $R$ is not locally injective.

On the other hand, if $u_0=u_1=u_2=u_3=0$ then by taking $\mathbb{C}^b=0$ and $\phi=0$ we get that $S$ is a subrepresentation of vector dimension $(1,0,0)$.
\end{proof}

Now we are able to characterize the representations in $\cali(1)$ in terms of the dimension vectors of its subrepresentations.

\begin{Prop}\label{fibrados}
Let $R$ be a representation in $\cali(1)$. Then the dimension vectors of its subrepresentations are exactly $(0,b,1)$ for all $b \in \{0,1,2,3,4\}$.
\end{Prop}
\begin{proof}
The possible dimension vectors of subrepresentations of a representation in $R_{\theta}(1)$ are of the kind $(0,b,0), (1,b,1), (0,b,1)$ and $(1,b,0)$. By Proposition \ref{existence0b1}, $R$ has subrepresentations of dimension vectors $(0,b,1)$ for all $b \in \{0,1,2,3,4\}$. As $R$ is both globally injective and globally surjective, by Theorem \ref{teorema}, it does not have subrepresentations of dimension vectors $(0,b,0), (1,b,1)$ and, by Proposition \ref{posto0}, it also does not have subrepresentations of dimension vectors $(1,b,0)$.   
\end{proof}


\section{Chamber decomposition for $\calr_\theta(1)$}\label{sec:modspc}

As the stability parameter $\theta=(\alpha,-(\alpha+\gamma)/4,\gamma)$ depends only on the values of $\alpha$ and $\gamma$, we can talk about $(\alpha,\gamma)$-stability. In this section, we obtain a wall-and-chamber decomposition of the real $\alpha\gamma$-plane of stability parameters.

In this setting a representation $R$ of dimension vector $(1,4,1)$ is $(\alpha,\gamma)$-stable if, and only if, every proper subrepresentation $S$ of dimension vector $(a,b,c)$ satisfies

$$\theta \cdot (a,b,c) < 0 \Leftrightarrow$$
$$(4a-b)\alpha + (4c-b)\gamma < 0$$

From Proposition \ref{existence0b1} and Proposition \ref{fibrados}, we know that every representation $R \in \calr_{\theta}(1)$ has subrepresentation $S$ of dimension vector $(0,b,1)$ for all $b \in \{0,1,2,3,4\}$ and the representations in $\cali(1)$ have exactly subrepresentations of this kind. 

Then for $R \in \cali(1)$ to be stable it is required that $\theta \cdot (0,b,1) < 0$ for every subrepresentation $S$ of dimension vector $(0,b,1)$, that is,

$$(4-b)\gamma < b\alpha$$
for $b \in \{0,1,2,3,4\}$.

The 5 possible values of $b$ give us $5$ inequalities whose intersection is the fourth quadrant of the real plane determined by $(\alpha,\gamma)$.

Thus for values of $(\alpha,\gamma)$ in the fourth quadrant we have $\cali(1) \subset \calr_{\theta}(1)$ and for values of $(\alpha,\gamma)$ outside the fourth quadrant we have $\calr_{\theta}(1)=\emptyset$, by Lemma \ref{empty}.

We are now interested in knowing which are exactly the orbits of representations in $\calr_{\theta}(1)\setminus\cali(1)$ for values of $(\alpha,\gamma)$ in the fourth quadrant.

\begin{Prop} \label{unstable} If $R \in R_{\theta}(1)$ is a globally (surjective) injective representation which is not locally (injective) surjective then $R$ is not $(\alpha,\gamma)$-stable for all values of $\alpha$ and $\gamma$.  

\end{Prop}
\begin{proof}
In either case, from Proposition \ref{posto0}, we know $R$ has subrepresentation $S$ of dimension vector $(1,b,0)$ with $\b \in \{0,1,2,3,4\}$. Then

$$\theta \cdot (1,b,0) < 0 \Leftrightarrow 4\alpha+b(-\alpha-\gamma)<0$$.

If $b=0$ then $\alpha <0$ and if $b>0$ then $\gamma > \frac{4-b}{b}\alpha$. In both cases the intersection with the fourth quadrant is empty and hence $R$ is not $(\alpha,\gamma)$-stable.

\end{proof}

\begin{Prop}\label{planeprop}
Let $(\alpha,\gamma)$ be a value in the fourth quadrant of the real plane. If $R$ is not globally injective then $R$ is $(\alpha,\gamma)$-stable only for $\gamma < -\alpha$. If $R$ is not globally surjective then $R$ is $(\alpha,\gamma)$-stable only for $\gamma >-\alpha$.
\end{Prop}
\begin{proof}
By Theorem \ref{teorema}, if $R$ is not globally injective then there is a subrepresentation $S$ of vector dimension $(1,b,1)$ with $\b \in \{0,1,2,3\}$. Consider $R$ $(\alpha,\gamma)$-stable. Thus

$$\theta \cdot (1,b,1) < 0 \Leftrightarrow (4-b)\alpha + (4-b)\gamma < 0$$

which implies $\gamma < -\alpha$ since $b \in \{0,1,2,3\}$.

Again by Theorem \ref{teorema}, if $R$ is not globally surjective then there is a subrepresentation $S$ of vector dimension $(0,b,0)$ with $\b \in \{1,2,3,4\}$. If $R$ is $(\alpha,\gamma)$-stable then

$$\theta \cdot (0,b,0) < 0 \Leftrightarrow b \frac{-\alpha-\gamma}{4}<0$$

which implies $\gamma>-\alpha$ since $b \neq 0$.
\end{proof}

\begin{Prop}
The moduli spaces associated to values of $(\alpha,\gamma)$ in the fourth quadrant of the real plane such that $\gamma < -\alpha$ are formed exactly by the globally surjective representations which are locally injective as the moduli spaces for $\gamma > -\alpha$ are formed exactly by the globally injective representations which are locally surjective.
\end{Prop}
\begin{proof}
Let $(\alpha,\gamma)$ be in the fourth quadrant and let $R$ be a representation whose orbit is in the moduli space associated to $(\alpha,\gamma)$.
By Proposition \ref{planeprop}, $R$ must be either globally injective or globally surjective. 
If $\gamma > -\alpha$ then again by Proposition \ref{planeprop}, $R$ must be globally injective and by Proposition \ref{unstable} it must be locally surjective. Analogously, if $\gamma < -\alpha$ then $R$ must be globally surjective and locally injective.
\end{proof}

Now we are going to prove that we can see the moduli spaces associated to $(\alpha,\gamma)$ in the fourth quadrant as compactifications of the open subset $\cali(1) \subset \calr_{\theta}(1)$, all of them isomorphic to $\mathbb{P}^5$.

Let $(\alpha,\gamma)$ be in the fourth quadrant such that $\gamma < -\alpha$. 


Let $R$ be a globally surjective representation of nontrivial rank (locally injective) in an fixed orbit of $\calr_{\theta}(1)$: 
\begin{center}
$R:$
\begin{tikzcd}
\mathbb{C} \arrow[r, "u_2" description , shift right=2]
  \arrow[r,"u_3" description, shift right=6]  
  \arrow[r, "u_1" description, shift left=2]
  \arrow[r, "u_0" description,  shift left=6]
& \mathbb{C}^{4} \arrow[r, "v_2" description , shift right=2]
  \arrow[r,"v_3" description, shift right=6]  
  \arrow[r, "v_1" description, shift left=2]
  \arrow[r, "v_0" description,  shift left=6]
& \mathbb{C}
\end{tikzcd}
\end{center}
Up to the action of a convenient $g \in G$ we know we can consider $\{v_0,v_1,v_2,v_3\}$ as being the canonical basis. In this case the representation $R$ is uniquely determined by the values of $\{u_0,u_1,u_2,u_3\}$ up to the multiplication of a nonzero scalar. 

Since $\{v_0,v_1,v_2,v_3\}$ is the canonical basis, from the relations of the quiver $\mathbf{Q}$ we get:

\begin{center}$u_0= \begin{bmatrix} 0\\a\\b\\c \end{bmatrix}$, $u_1= \begin{bmatrix} -a\\0\\d\\e \end{bmatrix}$, $u_2= \begin{bmatrix} -b\\-d\\0\\f \end{bmatrix}$, $u_3= \begin{bmatrix} -c\\-e\\-f\\0 \end{bmatrix}$\end{center}

Hence there is a bijective correspondence between orbits in $\calr_{\theta}(1)$ and nontrivial skew-symmetric matrices 

$$\tilde{R}=
\begin{bmatrix}
0 & -a & -b & -c\\
a & 0 & -d & -e\\
b & d & 0 & -f\\
c & e & f & 0\\
\end{bmatrix}$$

up to the multiplication of a nonzero scalar.

Thus there exists a bijective correspondence between orbits of $\calr_{\theta}(1)$ and points of $\mathbb{P}^5$ whose homogeneous coordinates can be represented by $[a:b:c:d:e:f]$, the entries of the skew symmetric matrix above.

Further, one can easily check that ${\rm det}(\tilde{R})=(be-af-dc)^2$. Since $\tilde{R}$ is skew-symmetric, there exists an invertible matrix $P$ such that $P^t\tilde{R}P$ is of the kind

$$\begin{bmatrix}
0 & \lambda_1 & 0 & 0\\
-\lambda_1 & 0 & 0 & 0\\
0 & 0 & 0 & \lambda_2\\
0 & 0 & -\lambda_2 & 0\\
\end{bmatrix}$$

and hence the rank of $\tilde{R}$ is 0, 2 or 4. We can not have ${\rm rank}(\tilde{R})=0$ since this would imply $u_0=u_1=u_2=u_3=0$, that is, the representation $R$ would not be locally injective.

We know $R$ is a globally surjective representation of rank 4 if and only if $R$ is also globally injective, that is, $R \in \cali(1)$. Thus we can identify $\cali(1)$ with the open set given by the complement of the quadric ${\rm det}(\tilde{R})=0$ in $\mathbb{P}^5$. On the other hand, the instanton sheaves in $\call(1)\setminus\cali(1)$ are in correspondence, by Proposition \ref{functor}, with the globally surjective representations of rank 2 and hence with the points in the quadric ${\rm det}(\tilde{R})=0$ in $\mathbb{P}^5$.

Similarly, if we take $(\alpha,\gamma)$ in the fourth quadrant such that $\gamma > -\alpha$ and we take $R$ a globally injective representation of nontrivial rank in an fixed orbit of $\calr_{\theta}(1)$ then we get again that $R_{\theta}(1)=\mathbb{P}^5$ is a compactification of the open set $\cali(1)$ which is the complement of a quadric. In this case the points of the quadric are in correspondence with the globally injective representations of rank 2 which can be seen as the perverse sheaves dual to the instanton sheaves in $\call(1)\setminus\cali(1)$ by Proposition \ref{functor}.

We have therefore completed the proof of the Main Thorem.



\bibliographystyle{amsalpha}

\end{document}